\documentclass[11pt,a4paper]{article}
\usepackage{amsmath,amssymb,amsthm,amsfonts,latexsym,graphicx,subfigure}
\usepackage[all,import]{xy}
\usepackage[numbers,sort&compress]{natbib}
\usepackage{indentfirst}
\usepackage{fancyhdr}
\usepackage{bbm,color,tikz} \usepackage[all]{xy}
\usetikzlibrary{shapes.geometric, arrows}
\usepackage{accents,cases,enumerate}
\usepackage{multirow}
\usepackage[lined, ruled, linesnumbered, longend]{algorithm2e}

\usepackage{array}
\newcommand{\PreserveBackslash}[1]{\let\temp=\\#1\let\\=\temp}
\newcolumntype{C}[1]{>{\PreserveBackslash\centering}p{#1}}
\newcolumntype{R}[1]{>{\PreserveBackslash\raggedleft}p{#1}}
\newcolumntype{L}[1]{>{\PreserveBackslash\raggedright}p{#1}}

\usepackage{setspace}
\makeatletter
\def\wbar{\accentset{{\cc@style\underline{\mskip8mu}}}}

\makeatother

\renewcommand{\vec}[1]{\mbox{\boldmath \small $#1$}}

\newcommand{\sss}{\mathcal{S}}

\theoremstyle{plain}
\newtheorem{thm}{Theorem}

\newtheorem{remark}{Remark}

\newtheorem{cor}{Corollary}
\newtheorem{pro}{Proposition}
\newtheorem{example}{Example}

\def\F{{\mathcal F}}

\newcommand{\R}{\mathbb{R}}

\def\F{{\mathcal F}}
\def\T{{\mathcal T}}
\begin{document}
\bibliographystyle{unsrt}
\title{Interactions between Hlawka Type-1 and Type-2 Quantities} 
\author{Xin Luo}
\renewcommand{\thefootnote}{\fnsymbol{footnote}}
\footnotetext{ Academy of Mathematics and Systems Science, Chinese Academy of Sciences, Beijing 100190, P.R. China. \\
Email address: {\tt xinluo@amss.ac.cn} \; and \;{\tt xinlnew@163.com} (Xin Luo).}

\date{}
\maketitle
\allowdisplaybreaks

\begin{abstract}
The classical Hlawka inequality possesses deep connections with zonotopes and zonoids in convex geometry, and has been related to Minkowski space. We introduce Hlawka Type-1 and Type-2 quantities, and establish a  Hlawka-type relation between them, which connects a vast number of strikingly different variants of the Hlawka inequalities, such as Serre's reverse Hlawka inequality in the future cone of the Minkowski space, the Hlawka inequality for subadditive function on abelian group by Ressel, and the integral analogs by Takahasi et al. Besides, we announce several enhanced results, such as the Hlawka inequality for the power of measure function. Particularly, we give a complete study of the Hlawka inequality for  quadratic form which relates to a work of Serre.
\\~\\
\textbf{Keywords}:  Hlawka's inequality, quadratic form, subadditivity
\end{abstract}

\section{Introduction}

 Hlawka's inequality saying for any $x,y,z$ in a inner product space
\begin{equation}\label{eq:Hlawka-classical}
\|x\|+\|y\|+\|z\|+\|x+y+z\|\ge \|x+y\|+\|y+z\|+\|z+x\|,
\end{equation}
 was proved firstly by Hlawka and originally appeared in 1942 in a paper of Hornich \cite{Hornich42}.
 It has a long series of investigations and extensions, such as the Hlawka inequality in integral form  \cite{TTW00,TTW09} and abelian group  \cite{R15}. The readers can also find an excellent summary of related works in \cite{Wl}, and the beautiful relations to discrete and convex geometry like zonotopes as well as zonoids by Witsenhausen  \cite{W78,W73,SW1983}.

Recently, Serre consider the pseudo-norm for the future cone of the Minkowski space \cite{Serre15}.  There he presented the reverse Hlawka-type inequality. 

According to these beautiful works, the classical Hlawka inequality has deep connections with zonotopes and zonoids in convex geometry, and relates to the geometry on the timelike cones of  Minkowski spaces. 

Note that the proof of \eqref{eq:Hlawka-classical} depends on the identity
\begin{equation}\label{eq:Hlawka-classical-2form}
\|x\|^2+\|y\|^2+\|z\|^2+\|x+y+z\|^2= \|x+y\|^2+\|y+z\|^2+\|z+x\|^2,
\end{equation}
which is an equality on quadratic forms. To some extend, the one-homogeneous inequality \eqref{eq:Hlawka-classical} essentially relates to the two-homogeneous equality \eqref{eq:Hlawka-classical-2form}.

In this work, we introduce Hlawka Type-1 and Type-2 quantities, and establish a Hlawka-type relation which encodes the signatures of them (see Theorem \ref{thm:operator-main}). 

This helps us to give the Hlawka inequality for a class of functions on semigroups. By this result, we can connect a vast number of Hlawka inequalities in the literature, even though they come from various perspectives and are very different from each other. Furthermore, we announce several exciting results, such as the Hlawka inequality on the power of measure function. Particularly, we investigate the Hlawka inequality on quadratic forms thoroughly.

For a glimpse of these results, we give some remarkable notes here:

\begin{itemize}
\item  For a reversed version of the Hlawka inequality, Serre gave a demonstration in the future cone of the Minkowski space \cite{Serre15}. In that paper, he shows:
if $q$ is a quadratic form on $\R^n$ with signature $(1, n-1)$, then the length $l=\sqrt{q}$
satisfies
\begin{equation}\label{eq:Serre}
l(x)+l(y)+l(z)+l(x+y+z)\leq l(x+y)+l(y+z)+l(z+x)
\end{equation}
for every vectors $x,y,z$ in the future cone with respect to $q$. In the present paper, we show a simple proof for \eqref{eq:Serre}, and give a systematic study for the Hlawka inequality on quadratic forms (see Section \ref{sec:quadratic}).

\item Ressel \cite{R15} shows a generalization of the Hlawka inequality for subadditive functions on abelian groups. In this work, we extend his result to the setting of  sub/super-additive functions on semigroups (see Section \ref{sec:semigroup}). For convenience, Ressel's result is provided in Example \ref{cor:R15} as an application. 
    
\item Takahasi et al \cite{TTW00,TTW09} study the integral analogs of the Hlawka inequality. In Section \ref{sec:integral}, we generalize this integral inequality to the form of positive linear operator, and their main theorem is rewritten in Example \ref{cor:TTW00}.
\end{itemize}

This paper provides a theorem combining the above different progresses together in a unified form (see Theorem \ref{thm:operator-main}), which also produces several other promotive results.

\section{ The Hlawka-type relation between Hlawka Type-1 and Type-2 quantities}\label{sec:Hlawka-relation}

\noindent \textbf{Basic settings}: 
\begin{itemize}
\item Given a nonempty set $\Omega$, let  $\mathbb{R}^\Omega$ be 
  the ring  of all real valued functions on $\Omega$  equipped with the the standard addition operator `$+$' and the standard product operator `$\cdot$'.
 \item Let $\sss\subset \mathbb{R}^\Omega$ be a sub-ring of real functions in $\mathbb{R}^\Omega$, i.e., $\sss$ is closed under both summation and  multiplication, as well as $1\in \sss$ and $\sss$ is also a real linear space. Here, we use $1$ to denote the constant function in $\sss$ satisfying $1(\omega)=1$, $\forall\omega\in\Omega$.
 \item Let $T: \sss \to \mathbb{R}$ be a linear and signature-preserving function  (i.e., for $\zeta\in \sss$ satisfying $\forall\omega\in\Omega,\,\zeta(\omega)\geq0$, there holds $T(\zeta)\geq0$).
\end{itemize}

From the basic settings, we can see that $\sss$ is a ring of  functions and it is  a real  linear space with the base consisted of some functions on $\Omega$, and $T$ is indeed a nonnegative linear functional on $\sss$. For convenience, $T$ can be regarded as a summation operator or an integral operator.

 \begin{thm}\label{thm:operator-main}
Given $\Omega,\sss,T$  in basic settings, $a,b\in \mathbb{R}$ with $a\ne0$ and  $a+b >0$,   for a nonempty set $X$, let $\eta, \xi:\Omega\to X$ and $f: X\rightarrow \mathbb{R}$  satisfy  $f\circ\xi,f\circ\eta\in \sss$ and  $(f\circ\xi)(\omega)-(f\circ\eta)(\omega) \leq b$, $\forall\omega\in\Omega$, where `$\circ$' represents the composition operator, we have the following:

\item (I) If $(f\circ\xi)(\omega)+(f\circ\eta)(\omega) \leq a$, $\forall\omega\in\Omega$, then $H_2\ge 0$ implies $H_1\ge 0$, and $H_1\le 0$ implies $H_2\le 0$;
\item(II) If $(f\circ\xi)(\omega)+(f\circ\eta)(\omega) \geq a$, $\forall\omega\in\Omega$, then $H_1\ge 0$ implies $H_2\ge 0$, and $H_2\le 0$ implies $H_1\le 0$.

\noindent Here
$$H_1=(T(1)-c)b+T(f\circ\eta)- T(f\circ\xi)$$
is called the {\sl Hlawka Type-1  quantity}, 
and
$$H_2=(T(1)-c)b^2+T(f^2\circ\eta) - T(f^2\circ\xi)$$
is called the   {\sl Hlawka Type-2 quantity},   in which $c:=\frac{2}{a}T(f\circ\eta)$.
 \end{thm}
 
 \begin{remark}
In many examples and applications (see Section \ref{sec:application}), $\xi$ is some function of $\eta$, and $f$ can be thought of as a norm. So,   to some extent,  $H_1$ can be seen as a 1-homogeneous
function of $\eta$.
 \end{remark}
 
In summary, Theorem \ref{thm:operator-main} says that under suitable `summation control' and `difference control', the signatures of Hlawka Type-1  quantity $H_1$ and Hlawka Type-2  quantity $H_2$ are essentially depended on each other in some way.

 \begin{proof}
 First, the relation among the quantities in Theorem \ref{thm:operator-main} can be shown in the following diagram:
   \begin{center}
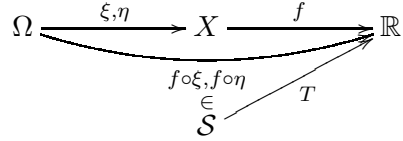
 \begin{spacing}{0.5}
$$
\xymatrix{
\Omega\ar[rr]^{\xi,\eta}\ar@/_1pc/[rrrr]_{ \scriptstyle\begin{array}{c}\scriptstyle f\circ \xi,f\circ\eta \\{ \scriptstyle \in}\end{array} }& & X\ar[rr]^f & &\mathbb{R}\\ & & {\sss \ar[rru]_{T} } & &
}
$$
\end{spacing}
\end{center}

We note the following identities: 
\begin{align*}
  &\;\;\;\;  T\left((a-f\circ\eta-f\circ\xi)(f\circ\eta+b-f\circ\xi)\right)
\\&= aT(f\circ\eta)-T(f^2\circ\eta)+T(f^2\circ\xi)+T\left((f\circ\eta) \cdot (f\circ\xi)\right)-T((f\circ\xi )\cdot (f\circ\eta))
\\&\;\;\;\; -aT(f\circ\xi)-bT(f\circ\eta)
-bT(f\circ\xi)+abT(1)
\\&= aT(f\circ\eta)+\left(T(1)-c\right)b^2+ T(f\circ\eta)b + (T(f^2\circ\xi) -T(f^2\circ\eta) -(T(1)-c)b^2)
\\&\;\;\;\; -T(f\circ\xi)(a+b)+(T(1)a-2T(f\circ\eta))b +T((f\circ\eta) \cdot (f\circ\xi))-T((f\circ\xi )\cdot (f\circ\eta))
\\&=  T(f\circ\eta)(a+b)+\left(T(1)-c\right)b^2 + (T(f^2\circ\xi) -T(f^2\circ\eta) -(T(1)-c)b^2)
\\&\;\;\;\; -T(f\circ\xi) (a+b)+(T(1)-c)ab
\\&= \left((T(1)-c) b+T(f\circ\eta) -T(f\circ\xi)\right)(a+b)
 + (T(f^2\circ\xi) -T(f^2\circ\eta) -(T(1)-c)b^2),
\end{align*}
where the notation $f^2\circ\eta:=(f\circ\eta)\cdot(f\circ\eta)$ is used. Therefore, we obtain
\begin{align}
  &\;\;\;\;  T\left((a-f\circ\eta-f\circ\xi)(f\circ\eta+b-f\circ\xi)\right) \notag
\\&= \left((T(1)-c) b+T(f\circ\eta) -T(f\circ\xi)\right)(a+b) - ((T(1)-c)b^2+T(f^2\circ\eta) -T(f^2\circ\xi)). \label{eq:T}
\end{align}

(I). For any $ \omega\in \Omega$, $f\circ\eta(\omega)+b\geq f\circ\xi(\omega)$ and $a\geq f\circ\eta(\omega)+f\circ\xi(\omega)$.
By the assumption,   $\forall \omega\in\Omega$,
$$(a-f\circ\eta(\omega)-f\circ\xi(\omega))(f\circ\eta(\omega)+b-f\circ\xi(\omega))\geq 0.$$
This deduces that
$$T\left((a-f\circ\eta-f\circ\xi)(f\circ\eta+b-f\circ\xi)\right)\geq0.$$
Accordingly, Eq.~\eqref{eq:T} gives
$$\left((T(1)-c)b+T(f\circ\eta) -T(f\circ\xi)\right)(a+b)\geq (T(1)-c)b^2+T(f^2\circ\eta) -T(f^2\circ\xi)$$
which arrives the final result.

(II). Since only the assumption $a\leq f\circ\eta(\omega)+f\circ\xi(\omega)$ is reversed, similar process gives
$$\left((T(1)-c)b+T(f\circ\eta) -T(f\circ\xi)\right)(a+b)\leq (T(1)-c)b^2+T(f^2\circ\eta) -T(f^2\circ\xi)$$
and then the reversed case could be verified immediately.
\end{proof}


\begin{remark}
From the proof of Theorem \ref{thm:operator-main}, it is obvious that the conditions could be weaken as follows:

Given $a\in \mathbb{R}\setminus\{0\}, b\in \mathbb{R}$ with $  a+b >0 $, for $\eta, \xi:\Omega\to X$ and $f: X\rightarrow \mathbb{R}$, let $\tilde{\Omega}=\{\omega\in\Omega|b+f\circ\eta(\omega)-f\circ\xi(\omega)\neq 0\}$.

 If  $a \geq f\circ\eta(\omega)+f\circ\xi(\omega),\; b \geq f\circ\xi(\omega)-f\circ\eta(\omega) $,  $\forall\omega\in
\tilde{\Omega}$, then
$H_2\ge 0$ $\Rightarrow$ $H_1\ge 0$, and $H_1\le 0$ $\Rightarrow$ $H_2\le 0$.

If $a \leq f\circ\eta(\omega)+f\circ\xi(\omega),\; b \geq f\circ\xi(\omega)-f\circ\eta(\omega)$,  $\forall\omega\in
\tilde{\Omega}$, then $H_1\ge 0$ $\Rightarrow$ $H_2\ge 0$, and $H_2\le 0$ $\Rightarrow$ $H_1\le 0$.
\end{remark}

\begin{remark}
To some extent, the two key controls for $f\circ\eta+f\circ\xi$ and $f\circ\xi-f\circ\eta$ via constants $a$ and $b$ are indeed `summation control' and `difference control'. The identities used in the proof of Theorem \ref{thm:operator-main} is inspired by product-to-sum formulas. 
\end{remark}

\vspace{0.5cm}
\section{Applications to variant Hlawka inequalities}\label{sec:application}
\subsection{Applications to quadratic forms}\label{sec:quadratic}
 Given  a nondegenerate quadratic form $q$, i.e., $q(x)=x^\top Qx$, where $x^\top$ is the transpose of the vector $x$ and $Q$ is a matrix of dimension $n$.
Henceforth a pair $(k,n-k)$ is said to be the signature of $Q$, if $Q$ has $k$ positive eigenvalues and $(n-k)$ negative eigenvalues.   Consider $l=\sqrt{q}$,
then we have the following:
\begin{pro}\label{quadratic}
(P1) If $Q$ is of the signature $(1,n-1)$, then $l$ satisfies the reversed
Hlawka inequality in the closure of the future cone;
(P2) If $Q$ is of the signature $(n,0)$, then $l$ satisfies the Hlawka inequality in $\mathbb{R}^n$.
\end{pro}
\begin{proof}
We will apply Theorem \ref{thm:operator-main} to this setting, where the symbols appearing in  Theorem \ref{thm:operator-main} can be concretely chosen (see Table~\ref{tab:pro1}). 

Firstly, according to the definition of $q$, there is
\begin{equation}\label{quadratic1}
q(x+y+z)+q(x)+q(y)+q(z)=q(x+y)+q(x+z)+q(y+z).
\end{equation}

(P1) Since $Q$ is of the signature $(1,n-1)$, we may assume without loss of generality that $Q=\mathrm{diag}(1,-1,\cdots,-1)$ and let $X=\{x=(x_1,\cdots,x_n)|q(x)>0,x_1> 0\}$, i.e., the future cone in Minkowski space.
\begin{table}
\centering
\caption{\small The concrete quantities of Theorem \ref{thm:operator-main} used in the proof of Proposition \ref{quadratic} (1). While, for Proposition \ref{quadratic} (2), we only let $X$ be replaced by $\R^n$.}
\begin{tabular}{|l|l|}
               \hline
             Terminologies in Theorem \ref{thm:operator-main} & Concrete choices in Proposition \ref{quadratic} (1) for fixed $x,y,z\in X$ \\
               \hline
                 $\Omega=$ & $\{1,2,3\}$\\
                 \hline
               $X=$ & $\{x=(x_1,\cdots,x_n)|q(x)>0,x_1> 0\}$\\
               \hline
               $\sss=$ & $\R^{\{1,2,3\}}$\\
               \hline
               $T=$&$\sum_{\omega\in\{1,2,3\}}$, i.e., $T(g)=g(1)+g(2)+g(3)$, $\forall g\in\sss$\\
               \hline
               $\eta=$&$x,y,z$ for $\omega=1,2,3$ respectively\\
               \hline
               $\xi=$&$\sum_{\omega=1}^3\eta(\omega)-\eta$, i.e., $y+z, z+x, x+y$ for $\omega=1,2,3$ respectively\\
               \hline
               $f=$&$\sqrt{q}$\\
               \hline
               $a=$&$\sqrt{q(x)}+\sqrt{q(y)}+\sqrt{q(z)}$\\
               \hline
               $b=$&$\sqrt{q(x+y+z)}$\\
               \hline
            \end{tabular}
             \label{tab:pro1}
\end{table}
Indeed, there is no subtraction `$-$' in $X$ and it is closed under addition.
Since $q(x)=x^2_1-x^2_2-\cdots-x^2_n>0, \ \ q(y)=y^2_1-y^2_2-\cdots-y^2_n>0$,  i.e.,
$$
x^2_1>x^2_2+\cdots+x^2_n, \ \  y^2_1>y^2_2+\cdots+y^2_n,
$$
by Cauchy inequality, the following inequality holds:
$$
x^2_1y^2_1>(x^2_2+\cdots+x^2_n)(y^2_2+\cdots+y^2_n)\geq (x_2y_2+\cdots+x_ny_n)^2.
$$
Due to $x,y\in X$, there is $x_1y_1> 0$, so $x_1y_1>x_2y_2+\cdots+x_ny_n$, i.e., $x^\top Qy=y^\top Qx=x_1y_1-x_2y_2-\cdots-x_ny_n>0$.
Hence $q(x+y)=q(x)+q(y)+x^\top Qy+y^\top Qx> 0$, which implies $x+y\in X$.

According to the Azteca inequality (i.e., a reversed version of Cauchy inequality), for any $x,y \in X$, there is
$$(x_1y_1-\sum_{j\geq2}x_jy_j)^2\geq (x^2_1-\sum_{j\geq2}x^2_j)(y^2_1-\sum_{j\geq2}y^2_j),$$
i.e.,

\begin{equation}\label{quadratic2}
(x^\top Qy)^2\geq x^\top Qx\cdot y^\top Qy.
\end{equation}
By further elementary computation, \eqref{quadratic2} is equivalent to
\begin{equation}\label{quadratic3}
\sqrt{q(x+y)}\geq\sqrt{q(x)}+\sqrt{q(y)}
\end{equation}
whenever $x,y\in X$. By \eqref{quadratic3}, for any $x,y,z\in X$, there is
$$\sqrt{q(x)}+\sqrt{q(y+z)}\geq a,\ \ \sqrt{q(y+z)}-\sqrt{q(x)}\leq b.$$

By the parameters shown in Table~\ref{tab:pro1}, we further have $c=2$ in Theorem \ref{thm:operator-main}, $H_2=q(x+y+z)+q(x)+q(y)+q(z)-q(x+y)-q(x+z)-q(y+z)=0$ (by Eq.~\eqref{quadratic1}) and
$$H_1=l(x)+l(y)+l(z)+l(x+y+z)- l(x+y)-l(y+z)-l(z+x).$$
According to Theorem \ref{thm:operator-main} (II), $H_1\le0$, thus
$$
l(x)+l(y)+l(z)+l(x+y+z)\leq l(x+y)+l(y+z)+l(z+x),
$$
whenever $x,y,z\in X$. By taking limits, one can find that the reversed
Hlawka inequality also holds on the boundary of the future cone.

\vspace{0.2cm}

(P2) If $Q$ is $(n,0)$, we may assume without loss of generality that $Q=\mathrm{diag}(1,1,\cdots,1)$ and let $X=\mathbb{R}^n$. In this case, the inner product $\langle x, y\rangle:=x^\top Qy$ satisfies Cauchy inequality, i.e., $(x^\top Qy)^2\leq x^\top Qx\cdot y^\top Qy$.
By elementary computation, there is $\sqrt{q(x+y)}\leq\sqrt{q(x)}+\sqrt{q(y)}$. From this, we have $$\sqrt{q(x)}+\sqrt{q(y+z)}\leq a\ \text{ and } \ \sqrt{q(y+z)}-\sqrt{q(x)}\leq b.$$
In case $a+b> 0$, similar to (P1), according to Theorem \ref{thm:operator-main} (I) and Eq.~\eqref{quadratic1}, we have $l(x)+l(y)+l(z)+l(x+y+z)\geq l(x+y)+l(y+z)+l(z+x)$.
In the case of $a=0$ or $a+b=0$, i.e., $x=y=z=0$, it is obvious that $l(x)+l(y)+l(z)+l(x+y+z)= l(x+y)+l(y+z)+l(z+x)=0$.
Consequently, $l$ satisfies the Hlawka inequality.
\end{proof}

Proposition \ref{quadratic} contains Hlawka-type inequalities in the settings of both Euclidean case and Minkowski case. Moreover, by using Theorem \ref{thm:operator-main}, here we indeed provide an alternative and much  easier proof of the reverse Hlawka inequality in Minkowski space (Theorem 1.1 in \cite{Serre15}).

However, there is no similar conclusion on other cases that $Q$ is of the signature $(k,n-k)$ for $2\le k\le n-1$, and we will give an example to show this.

\begin{example}
If $Q$ is of the signature $(k,n-k)$ for $2\le k\le n-1$, we may assume without loss of generality that $Q=\textrm{diag}(\mathop{\underbrace{1,\cdots,1}}\limits_k,\mathop{\underbrace{-1,\cdots,-1}}\limits_{n-k})$. 
By finding suitable cone $X\subset \{x=(x_1,\cdots,x_n)\in\R^n|\,q(x)>0\}$, one may obtain that both the Hlawka inequality and the reversed Hlawka inequality fail for $l$.
Indeed, take $0<\epsilon \ll 1$, let
$$
\vec v_1=\vec v_2=(\mathop{\underbrace{1,1,\epsilon,\cdots,\epsilon}}\limits_k,
\mathop{\underbrace{1,\epsilon,\cdots,\epsilon}}\limits_{n-k})
,\ \
\vec v_3=(1,1,\epsilon,\cdots,\epsilon)
$$
and
$$
\vec v_4=(2,1,\epsilon,\cdots,\epsilon),\; \vec v_5=(1,2,\epsilon,\cdots,\epsilon).
$$
Consider $X=\{t_1\vec v_1+t_2\vec v_2+t_3\vec v_3+t_4\vec v_4+t_5\vec v_5|t_i>0, 1\leq i \leq 5\}$. It is clear that $X\subset\{x|\,q(x)>0\}$.
By computation, $l(\vec v_1)+l(\vec v_2)+l(\vec v_3)+l(\vec v_1+\vec v_2+\vec v_3)< l(\vec v_1+\vec v_2)+l(\vec v_2+\vec v_3)+l(\vec v_3+\vec v_1)$.
While, $l(\vec v_5)+l(\vec v_4)+l(\vec v_3)+l(\vec v_3+\vec v_4+\vec v_5)> l(\vec v_4+\vec v_5)+l(\vec v_3+\vec v_5)+l(\vec v_3+\vec v_4)$.
Thus, in $X$,  both the Hlawka inequality and the reversed Hlawka inequality do not hold for $l$.
\end{example}

\hspace{0.5cm}

\subsection{Applications to sub/super -additive functions on semigroups}\label{sec:semigroup}
Let $X$ in Theorem \ref{thm:operator-main} be an  abelian semigroup $(G,+)$, and let $F:G\rightarrow \mathbb{R}$ be a non-negative real-valued function.
We will consider the Hlawka inequality of the form 
\begin{equation}\label{eq:2^k}
F(x+y)^{1/2^k}+F(y+z)^{1/2^k}+F(z+x)^{1/2^k}\le F(x)^{1/2^k}+F(y)^{1/2^k}+F(z)^{1/2^k}+F(x+y+z)^{1/2^k},
\end{equation}
$\forall x,y,z\in G$, where $k$ is an integer.

\begin{pro}\label{pro:2^k}
Let $G$ be an abelian semigroup, and let  $x\mapsto F(x)$  be a non-negative real-valued function on $G$.

If $F$ is {\sl strong subadditive} (i.e.,  $F(x)+F(y)\ge F(x+y)$ and $F(x)+F(x+y)\ge F(y)$, $\forall x,y\in G$), and \eqref{eq:2^k} holds for some  $k_0\ge -1$, then \eqref{eq:2^k} holds for all $k\ge k_0$.

If $F$ is assumed to be superadditive  (i.e.,  $F(x)+F(y)\le F(x+y)$, $\forall x,y\in G$), and \eqref{eq:2^k} holds for some $k_0\le 0$, then \eqref{eq:2^k} holds for all $k\le k_0$.
\end{pro}

\begin{proof}
 Given $a,b>0$, the function $(a^t+b^t)^{1/t}$ is decreasing on $(0,\infty)$.
\begin{table}
\centering
\caption{\small The concrete quantities of Theorem \ref{thm:operator-main} used in the proof of Proposition \ref{pro:2^k}.} 
\begin{tabular}{|l|l|}
               \hline
             Terminologies in Theorem \ref{thm:operator-main} & Concrete choices in Proposition \ref{pro:2^k} for fixed $x,y,z\in X$ \\
               \hline
                 $\Omega=$ & $\{1,2,3\}$\\
                 \hline
               $X=$ &  abelian semigroup $G$\\
               \hline
               $\sss=$ & $\R^{\{1,2,3\}}$\\
               \hline
               $T=$& $\sum_{\omega\in\{1,2,3\}}$, i.e., $T(g)=g(1)+g(2)+g(3)$, $\forall g\in\sss$\\
               \hline
               $\eta=$& $x,y,z$ for $\omega=1,2,3$ respectively\\
               \hline
               $\xi=$&$\sum_{\omega=1}^3\eta(\omega)-\eta$, i.e., $y+z, z+x, x+y$ for $\omega=1,2,3$ respectively\\
               \hline
               $f=$& $F^{\frac 1{2^{k}}}$\\
               \hline
               $a=$&$F(x)^{\frac 1{2^{k}}}+F(y)^{\frac 1{2^{k}}}+F(z)^{\frac 1{2^{k}}}$\\
               \hline
               $b=$&$F(x+y+z)^{\frac 1{2^{k}}}$\\
               \hline
            \end{tabular}
             \label{tab:pro2}
\end{table}

Case (1). $F$ is non-negative and strong  subadditive.
For any $0<\alpha\le 1$,
$$
F(x+y)^\alpha\leq (F(x)+F(y))^\alpha\leq F(x)^\alpha+F(y)^\alpha.
$$
Suppose \eqref{eq:2^k} holds for some $k_0\ge -1$. Then for any $k> k_0$, and any $x,y,z\in G$,
$$
F(x+y)^{\frac 1{2^{k}}}+F(z)^{\frac 1{2^{k}}}\leq F(x)^{\frac 1{2^{k}}}+F(y)^{\frac 1{2^{k}}}+F(z)^{\frac 1{2^{k}}}
$$
and
$$
F(x+y)^{\frac 1{2^{k}}}-F(z)^{\frac 1{2^{k}}}\leq F(x+y+z)^{\frac 1{2^{k}}}.
$$
Here, let $a$ and $b$ in Theorem \ref{thm:operator-main} be $F(x)^{\frac 1{2^{k}}}+F(y)^{\frac 1{2^{k}}}+F(z)^{\frac 1{2^{k}}}$ and $F(x+y+z)^{\frac 1{2^{k}}}$,
respectively. The detailed parameters are shown in Table~\ref{tab:pro2}. 
If $a\neq 0$ and $a+b>0$, then the proof is finished by Theorem \ref{thm:operator-main} (I).
If $a=0$, then $F(x)=F(y)=F(z)=F(x+y)=F(x+z)=F(y+z)=F(x+y+z)=0$ and \eqref{eq:2^k} is obvious.

\vspace{0.3cm}

\noindent Case (2). $F$ is non-negative and  superadditive, i.e., $F(x)+F(y)\leq F(x+y)$.

Note that for any $\alpha \geq 1$,
$$
F(x+y)^\alpha\geq(F(x)+F(y))^\alpha\geq F(x)^\alpha+F(y)^\alpha.
$$
In consequence, for any $k\leq0$ and any $x,y,z\in G$, we have
$$
F(x+y)^{\frac 1{2^{k}}}+F(z)^{\frac 1{2^{k}}}\geq F(x)^{\frac 1{2^{k}}}+F(y)^{\frac 1{2^{k}}}+F(z)^{\frac 1{2^{k}}}
$$
and
$$
 F(x+y+z)^{\frac 1{2^{k}}}\geq F(x+y)^{\frac 1{2^{k}}}+F(z)^{\frac 1{2^{k}}}\geq F(x+y)^{\frac 1{2^{k}}}-F(z)^{\frac 1{2^{k}}}.
 $$
 If $a\neq 0$ and $a+b>0$, by Theorem \ref{thm:operator-main}, the result is proved.
 If $a+b=0$, then $F(x)=F(y)=F(z)=F(x+y)=F(x+z)=F(y+z)=F(x+y+z)=0$, the result is obvious.
 If $a=0$, then $F(x)=F(y)=F(z)=0$. According to the condition, we have
 $$ F(x+y)^{1/2^{k}}+F(y+z)^{1/2^{k}}+F(z+x)^{1/2^{k}}\le F(x+y+z)^{1/2^{k}}$$
 for some $k\leq 0$.
 Taking the square of above inequality, there is
 $$
 F(x+y)^{1/2^{k-1}}+F(y+z)^{1/2^{k-1}}+F(z+x)^{1/2^{k-1}}\le F(x+y+z)^{1/2^{k-1}}.
 $$
 Hereto, the prove is completed.
\end{proof}

Now we show an interesting example even though this result seems to be elementary.
\begin{example}
Taking $G=L^p$ and $F=\|\cdot\|_p$, together with Corollary 2.1 in \cite{W73} and Proposition \ref{pro:2^k}, we have
$$\|a+b\|_p^{\frac 1{2^k}}+\|b+c\|_p^{\frac 1{2^k}}+\|c+a\|_p^{\frac 1{2^k}}\le \|a\|_p^{\frac 1{2^k}}+\|b\|_p^{\frac 1{2^k}}+\|c\|_p^{\frac 1{2^k}}+\|a+b+c\|_p^{\frac 1{2^k}}$$
for any $a,b,c\in L^p$, and $k\in \mathbb{N}$, where $1\le p\le 2$.

Replacing $a,b,c$ respectively by $a^{2^k}$, $b^{2^k}$, $c^{2^k}$, one gets
$$\|(a^{2^k}+b^{2^k})^{\frac 1{2^k}}\|_{2^kp}+\|(b^{2^k}+c^{2^k})^{\frac 1{2^k}}\|_{2^kp}+\|(c^{2^k}+a^{2^k})^{\frac 1{2^k}}\|_{2^kp} \le \|a\|_{2^kp}+\|b\|_{2^kp}+\|c\|_{2^kp}+\|(a^{2^k}+b^{2^k}+c^{2^k})^{\frac 1{2^k}}\|_{2^kp}.$$

For convenience, we define an operation $\Diamond_k$ by $a\Diamond_k b=(a^{2^k}+b^{2^k})^{\frac 1{2^k}}$ for $1\le k<+\infty$, $a\Diamond_0 b:= a+b$ and $a\Diamond_\infty b:=|a|\vee |b|:=\max\{|a|,|b|\}$. Then using this notation, we obtain
$$\|a \Diamond_k b\|_{2^kp}+\|b\Diamond_k c\|_{2^kp}+\|c\Diamond_k a\|_{2^kp} \le \|a\|_{2^kp}+\|b\|_{2^kp}+\|c\|_{2^kp}+\|a\Diamond_k b\Diamond_k c\|_{2^kp}$$
for any $p\in[1,2]$ and any $k\in \mathbb{N}\cup\{+\infty\}$. Thus
$$\|a\Diamond_k b\|_p+\|b\Diamond_k c\|_p+\|c\Diamond_k a\|_p \le \|a\|_p+\|b\|_p+\|c\|_p+\|a\Diamond_k b\Diamond_k c\|_p$$
holds for $p\in[2^k,2^{k+1}]$, and by taking $k\to+\infty$, we have
$$
\|\max\{|a|,|b|\}\|_\infty+\|\max\{|b|,|c|\}\|_\infty+\|\max\{|c|,|a|\}\|_\infty\le \|a\|_\infty+\|b\|_\infty+\|c\|_\infty+\|\max\{|a|,|b|,|c|\}\|_\infty
.$$

\end{example}

A direct application of Proposition \ref{pro:2^k} is the following Hlawka inequality on abelian groups.

\begin{example}[Theorem 2 in \cite{R15}]\label{cor:R15}
Let $G$ be an abelian group,  $x\mapsto |x|$ a non-negative symmetric and subadditive function on $G$ (i.e., $|-x|=|x|$ and $|x|+|y|\ge|x+y|$, $\forall x,y\in G$), and let $S:[0,\infty)\to [0,\infty)$ be concave. Then, if $\forall x,y,z\in G$,
$$S^2(|x+y|)+S^2(|y+z|)+S^2(|z+x|)\le S^2(|x|)+S^2(|y|)+S^2(|z|)+S^2(|x+y+z|),$$
so does $S$.

In fact, the function $F(\cdot):=S(|\cdot|)$ must be non-negative and strong subadditive. So, Proposition \ref{pro:2^k} is applicable here.
\end{example}

The following measure-type Hlawka inequality is non-trivial and it cannot be deduced from Theorem 2 in \cite{R15} (i.e., Example \ref{cor:R15} above), because a measure space equipped with any set operation is not a group. But it can be obtained straightforward by Proposition \ref{pro:2^k} since a measure space with any set operation becomes a semigroup.
\begin{example}
Let $G$ be a measure space and $F=\mu$ be the measure. For the case of $k=0$, note that $\mu(A)+\mu(B)+\mu(C)-\mu(A\cup B)-\mu(B\cup C)-\mu(C\cup A)+\mu(A\cup B\cup C)=\mu(A\cap B\cap C)\ge 0$ and for symmetric difference $\triangle$,
$\mu(A)+\mu(B)+\mu(C)-\mu(A\triangle B)-\mu(B\triangle C)-\mu(C\triangle A)+\mu(A\triangle B\triangle C)=3\mu(A\cap B\cap C)\ge0$.

According to Proposition \ref{pro:2^k}, we have for any $k\ge 0$,
$$
\mu(A)^{1/2^k}+\mu(B)^{1/2^k}+\mu(C)^{1/2^k}+\mu(A\cup B\cup C)^{1/2^k}\ge \mu(A\cup B)^{1/2^k}+\mu(B\cup C)^{1/2^k}+\mu(C\cup A)^{1/2^k}
$$
and
$$
\mu(A)^{1/2^k}+\mu(B)^{1/2^k}+\mu(C)^{1/2^k}+\mu(A\triangle B\triangle C)^{1/2^k}\ge \mu(A\triangle B)^{1/2^k}+\mu(B\triangle C)^{1/2^k}+\mu(C\triangle A)^{1/2^k}
$$
because $
\mu(A)+\mu(B)\ge \mu(A\cup B)\ge \mu(A\triangle B)$ and $\mu(A)+\mu(A\cup B)\ge \mu(A)+\mu(A\triangle B)\ge \mu(B)$.
\end{example}

\vspace{0.5cm}

\subsection{Applications to integral form}\label{sec:integral}
Next, we would pay our attention to the following setting.
 Let $\Omega$ be a nonempty set and let $G$ be an abelian group, and let $x\mapsto |x|$ be a non-negative symmetric and subadditive function on $G$ (i.e., $|-x|=|x|$ and $|x|+|y|\ge|x+y|$, $\forall x,y\in G$). The function spaces $G^\Omega$ and $\mathbb{R}^\Omega$ are also abelian groups under the natural  operation `$+$'. Take an abelian subgroup $\F\subset G^\Omega$ and a linear subspace $\widehat{\F}\subset \mathbb{R}^\Omega$ equipped with  $T:\widehat{\F}\to\mathbb{R}$ satisfying the \textbf{basic settings} in the beginning of Section \ref{sec:Hlawka-relation}. Moreover,\footnote{Here, for $f\in \F$, $|f|$ is a function mapping $\Omega$ to $[0,\infty)$.} $\forall f\in \F$, $|f|\in \widehat{\F}$, $1\in \widehat{\F}$.

\begin{table}
\centering
\caption{\small The concrete quantities of Theorem \ref{thm:operator-main} used in the proof of Proposition \ref{thm:groupmain}. }
\begin{tabular}{|l|l|}
               \hline
             Terminologies in Theorem \ref{thm:operator-main} & Concrete choices in Proposition \ref{thm:groupmain} \\
                 \hline
               $X=$ & abelian group $G$\\
               \hline
               $\sss=$ & $\widehat{\F}$\\
               \hline
               $\eta=$& $\hat{g}$ \\
               \hline
               $\xi=$& $\mathcal{T}g-\hat{g}$\\
               \hline
               $f=$& $S|\cdot|$\\
               \hline
               $a=$& $A$\\
               \hline
               $b=$& $T(S|\hat{g}|)$\\
               \hline
            \end{tabular}
             \label{tab:pro3}
\end{table}

Applying Theorem \ref{thm:operator-main} to the above restricted situations, we have:

 \begin{pro}\label{thm:groupmain} 
 Given $A\neq 0 \in \R$, an operator $\T: \F\to G$, and two maps  $g,\hat{g}\in \F$, let $S:[0,+\infty)\to [0,+\infty)$ be a concave function such that $A>0$ or $S(|\T g|)>0$ and when $x$ satisfies $S|\hat{g}(x)|+S|\T g|\ne S|\hat{g}(x)-\T g|$, there is $A\ge S|\hat{g}(x)|+S|\T g-\hat{g}(x)|$. Then, 
      \begin{equation}\label{eq:HlawkaS2}
  \left(T(1)-C\right)S^2|\T g|+T(S^2|\hat{g}|)\ge T(S^2|\hat{g}-\T g|)
  \end{equation}
  implies
    \begin{equation}\label{eq:Hlawka}
  \left(T(1)-C\right)S|\T g|+T(S|\hat{g}|)\ge T(S|\hat{g}-\T g|),
  \end{equation}
  where $C=2T(S|\hat{g}|)/A$.
 \end{pro}
\begin{proof}
Taking $\xi=\mathcal{T}g-\hat{g}$, $\eta=\hat{g}$ and $f(\cdot)=S|\cdot|$ in Theorem \ref{thm:operator-main} (see Table~\ref{tab:pro3} for details), we immediately complete the proof.
\end{proof}

 The main theorems in  \cite{TTW00,R15} can be seen as direct conclusions of Proposition  \ref{thm:groupmain}.

\begin{proof}[A proof of Example \ref{cor:R15} (i.e., the main theorem in \cite{R15}) via Proposition \ref{thm:groupmain}]
Take $\Omega=\{1,2,3\}$, and for given $x,y,z\in G$, let $\hat{g}=g$ be defined as $g(1)=x$, $g(2)=y$ and $g(3)=z$. 
Let $T(S|g|)=S|g(1)|+S|g(2)|+S|g(3)|$ and $\T g=g(1)+g(2)+g(3)$ in Proposition \ref{thm:groupmain}. Then $T(1)=3$, $\T g-g(i)=g(j)+g(k)$, where $\{i,j,k\}=\{1,2,3\}$. Hence, the result is easy to check.
\end{proof}

\vspace{0.3cm}

Given an inner product space $(H,\langle\cdot,\cdot\rangle)$, suppose that  $\F\subset H^\Omega$ and $\widehat{\F}\subset \mathbb{R}^\Omega$ are linear spaces equipped with linear operators $\T: \F\to H$ and $T:\widehat{\F}\to\mathbb{R}$.  Then we have the following:

\begin{cor}\label{pro:inner}
Let $\T,T$ and  $f\in \F$ be such that  $T(|f|)>0$ and for any $a\in H$, there is $T\langle f,a\rangle=\langle \T f,a\rangle$, where $|\cdot|$ is the norm induced by the inner product. If      $T(|f|)\ge |f(x)|+|\T f-f(x)|$ whenever $x$ satisfies $-f(x)\ne \alpha \T f$ for any $\alpha\ge 0$, then the following holds:
  $$\left(T(1)-2\right)|\T f|+T(|f|)\ge T(|f-\T f|).$$
\end{cor}

\begin{proof}
By the basic properties on  inner products, we have
\begin{align*}
T\left(|f-\T f|^2\right)&=T\left(|f|^2+|\T f|^2-2\langle f,\T f\rangle\right)
\\ ~&=T(|f|^2)+|\T f|^2T(1)-2\langle \T f,\T f\rangle
\\ ~&=T(|f|^2)+|\T f|^2(T(1)-2).
\end{align*}
Let $S$ in Proposition \ref{thm:groupmain} be the identity operator and
 the rest conditions in Proposition \ref{thm:groupmain} are easy to be verified. The prove is completed.
\end{proof}

It is clear that $T$ and $\T$ are uniquely determined by each other according to Riesz's  representation theorem.

\begin{cor}\label{pro:inner2integral}
Let $(\Omega,\mu)$ be a finite measurable space and let $(H,\|\cdot\|)$ be an inner product space. Suppose $f,g:\Omega\to H$ are two nonzero integrable functions satisfying
$$\frac{\int_\Omega fd\mu}{\int_\Omega \|f\|d\mu}=\frac{\int_\Omega gd\mu}{\int_\Omega \|g\|d\mu}.$$
Assume that for $x$ with $-g(x)\ne\alpha \int_\Omega fd\mu$ for any $\alpha\ge0$, there is $$\int_\Omega \|f\|d\mu\ge \|g(x)\|+\left\|g(x)-\int_\Omega fd\mu\right\|.$$
Then we have the following Hlawka  inequality
  $$\left(\mu(\Omega)-C\right)\left\| \int_\Omega fd\mu\right\|+\int_\Omega \|g\|d\mu\ge \int_\Omega\left\|g -\int_\Omega fd\mu\right\|d\mu,$$
  where $C=2\int_\Omega \|g\|d\mu/\int_\Omega \|f\|d\mu$.
\end{cor}
\begin{proof}
Take $\T f= \int_\Omega f(\omega)d\mu$ and $T (\|f\|) = \int_\Omega\|f(t)\|d\mu(t)$. Now it is ready to apply  Proposition \ref{thm:groupmain} to complete the proof.
\end{proof}

\begin{cor}\label{cor:t=lambda}
Suppose that for $x$ with $- f(x)\ne\alpha \int_\Omega fd\mu$ for any $\alpha\ge0$, there is $$\int_\Omega \|f\|d\mu\ge t\| f(x)\|+\left\| t f(x)-\int_\Omega fd\mu\right\|$$
for some $t\ge 0$.
Then we have the following Hlawka-type  inequality
  $$\left(\mu(\Omega)-2t\right)\left\| \int_\Omega fd\mu\right\|+t\int_\Omega \|f\|d\mu\ge \int_\Omega\left\| t f -\int_\Omega fd\mu\right\|d\mu.$$
\end{cor}

\begin{cor}\label{bar}
 If $\bar f $ is a rearrangement of $f$ with the same distribution, and for $x$ with $-\bar f(x)\ne\alpha \int_\Omega fd\mu$ for any $\alpha\ge0$,  there is $$\int_\Omega \|f\|d\mu\ge \|\bar f(x)\|+\left\|\bar f(x)-\int_\Omega fd\mu\right\|.$$
Then we have the following Hlawka  inequality
  $$\left(\mu(\Omega)-2\right)\left\| \int_\Omega fd\mu\right\|+\int_\Omega \|\bar f\|d\mu\ge \int_\Omega\left\|  \bar f -\int_\Omega fd\mu\right\|d\mu.$$
\end{cor}
\begin{proof}
Clearly, the properties of the rearrangement imply that $ \int_\Omega fd\mu=\int_\Omega \bar fd\mu$ and $\int_\Omega \| f\|d\mu=\int_\Omega \|\bar f\|d\mu$. Hence, Corollary \ref{pro:inner2integral} is applicable here.
\end{proof}

Theorem 1 in  \cite{TTW00} could be viewed as a consequence  of Corollary \ref{bar}. In fact, taking $\bar f=f$ in Corollary \ref{bar}, it is easy to verify the following.
\begin{example}[Theorem 1 in \cite{TTW00}]\label{cor:TTW00}
Let $H$ be a Hilbert space, $(\Omega,\mu)$ be a finite measure space and let $f$ be a Bochner integrable $H$-valued function on $(\Omega,\mu)$. Suppose that
$$\int_\Omega\|f(t)\|d\mu(t)\ge \left\|f(\omega)-\int_\Omega f(t)d\mu(t)\right\|+\|f(\omega)\|\;\;(a.e., \omega\in\Omega_f),$$  where $\Omega_f=\{\omega\in\Omega:-f(\omega)\ne \alpha \int_\Omega f(t)d\mu(t)\text{ for any }\alpha\ge 0\}$.
Then
$$(\mu(\Omega)-2)\left\|\int_\Omega f(\omega)d\mu\right\|+\int_\Omega\|f(\omega)\|d\mu\ge \int_\Omega\left\|f(\omega)-\int_\Omega fd\mu\right\|d\mu.$$
\end{example}

Next remark contains some interesting examples as corollaries of Proposition  \ref{thm:groupmain}.
\begin{remark}Given  an inner product space $H$, we have:
\begin{itemize}
\item  For any $\lambda\in[0,1]$, $x,y,z\in H$,
$$(1-\lambda)(\|x\|+\|y\|+\|z\|)+(1+2\lambda)\|x+y+z\|\ge \|\lambda x+y+z\|+\|x+ \lambda y+z\|+\|x+\lambda y+z\|.$$
{\sl It is deduced by taking $\Omega=\{1,2,3\}$ and $t=(1-\lambda)$ in Corollary \ref{cor:t=lambda}, which is rather  different from Corollary 2 in  \cite{TTW00}.}
\item Let $\mu_i,\lambda\ge 0$ be such that $\sum_{i=1}^n\mu_i\|x_i\|\ge \lambda\mu_i\|x_i\|+\|\lambda x_i-\sum_{j=1}^n\mu_j x_j\|$ for any $1\leq i \leq n$. Then
    $$\left(\sum_{i=1}^n\mu_i-2\lambda\right)\left\|\sum_{i=1}^n\mu_ix_i\right\|+
    \lambda\sum_{i=1}^n\mu_i\|x_i\|\ge \sum_{i=1}^n\mu_i\left\|\lambda x_i-\sum_{j=1}^n\mu_jx_j\right\|.$$
{\sl It is deduced by taking $\Omega=\{1,\cdots,n\}$ and $\mu(i)=\mu_i$ in Corollary \ref{cor:t=lambda}, which is an improved version of Corollary 2 in  \cite{TTW00} and Proposition 11 in \cite{TTW09}.}
\end{itemize}
\end{remark}

\vspace{1cm}

{\bf Acknowledgements.} This research is supported by grant from the Project
funded by China Postdoctoral Science Foundations (No. 2019M660829). The author thanks her husband for interesting discussions. I am grateful to the anonymous referee for the comments and suggestions that
have greatly improved the first version of this paper.

 \end{document}